\documentclass{amsart}
\usepackage{amssymb}
\usepackage{graphicx,amssymb,mathrsfs}
\usepackage{pstricks,pst-node}
\newtheorem{thm}{Theorem}
\newtheorem{lemma}[thm]{Lemma}
\newtheorem{corollary}[thm]{Corollary}
\newtheorem{proposition}[thm]{Proposition}
\numberwithin{thm}{section} \theoremstyle{definition}
\newtheorem{defn}[thm]{Definition}
\newcommand\Hecke{\mathscr{H}}
\newcommand\Alg{\mathcal{A}}
\newcommand\Z{\mathbb{Z}}
\DeclareMathOperator{\Ind}{Ind} 
\begin{document}
\title{Kazhdan-Lusztig basis for generic Specht modules}
\author{Yunchuan Yin}
\address{Department of Applied Mathematics, Shanghai University of Finance and Ecomonics\\
Shanghai 200433, P. R. China} \email{yunchuan228@hotmail.com}\email{yyin@mail.shufe.edu.cn}
 \subjclass[2000]{Primary and secondary
20C08} \keywords{Hecke algebra, Coxeter group, Kazhdan-Lusztig
bases, Specht module, Murphy basis, $W\!$-graph}
\begin{abstract}
In this paper, we let $\Hecke$ be the Hecke algebra associated with a finite
Coxeter group $W$ and with one-parameter, over the ring of scalars $\Alg=\mathbb{Z}(q, q^{-1})$.  With an elementary method, we introduce a cellular basis of $\Hecke$ indexed by the sets $E_J (J\subseteq S)$ and obtain a general theory of "Specht modules". Our main purpose is to provide an algorithm for  $W\!$-graphs for the "generic Specht module", which associates with the Kazhdan and Lusztig  cell ( or more generally, a union of cells of $W$ ) containing the longest element of a parabolic subgroup $W_J$ for appropriate $J\subseteq S$. As an example of applications, we show a construction of $W\!$-graphs for the Hecke algebra of type $A$.
\end{abstract}
\maketitle
\section*{Preliminaries}
Let $W$ be a finite Coxeter group with $S$ the set of simple
reflections, and let $\Hecke$ be the corresponding Hecke algebra. We
use a variation of the definition given in~\cite{KL}, taking
$\Hecke$ to be an algebra over $\Alg=\Z[q^{-1},q]$, the ring of
Laurent polynomials with integer coefficients in the indeterminate
$q$. Then $\Hecke$ is a algebra generated by $(T_s)_{s\in S}$ subject to
\begin{align*}
{T_s}^2 &= 1+(q-q^{-1})T_{s}\\
\underbrace{T_rT_sT_r\cdots}_{\text{$m_{rs}$ factors}}
&= \underbrace{T_sT_rT_s\cdots}_{\text{$m_{rs}$ factors}}
\end{align*}
(for all $r,s \in S$).

Moreover, $\Hecke$ has $\Alg$-basis $\{\,T_w\mid w\in W\,\}$ where
$T_w = T_{s_1}T_{s_2} \cdots T_{s_l}$ whenever $s_1s_2 \cdots s_l$ is a reduced expression for $w$,  and
\begin{equation}\label{eq:1}
T_sT_w=\begin{cases}
T_{sw}&\text{if $\ell(sw)>\ell(w)$}\\[3 pt]
T_{sw}+(q-q^{-1})T_w&\text{if $\ell(sw)<\ell(w)$,}
\end{cases}
\end{equation}
for all $w\in W$ and $s\in S$. We also define $\Alg^+=\Z[q]$, the
ring of polynomials in $q$ with integer coefficients, and let
$a\mapsto\overline{a}$ be the involutory automorphism of $\Alg$
such that $\overline{q}=q^{-1}$. This involution on $\Alg$ extends
to an involution on $\Hecke$ satisfying
$\overline{T_s}=T_s^{-1}=T_s+(q^{-1}-q)$ for all $s\in S$. This
gives $\overline{T_w}=T_{w^{-1}}^{-1}$ for all $w\in W$.
The map $\Hecke\rightarrow \Hecke, h\longmapsto \overline{h}$ is a ring involution such that
\[
\overline{\sum_{w\in W}a_wT_w}=\sum_{w\in W}\bar{a_w}{T_{w^{-1}}}^{-1}, a_w\in \Alg.
\]
\subsection{Kazhdan-Lusztig basis}
There are two types of Kazhdan-Lusztig bases of $\Hecke$, denoted by $\{C_w |w\in W\}$ and $\{C'_w |w\in W\}$ in the original article by Kazhdan-Lusztig~\cite{KL}.
It will be technically more convenient to work with the $C$-basis. The reason can be seen, for example, in Lusztig~\cite[chap.18]{Lusztig}.
The basis element
$C_w$ is uniquely determined by the conditions that
 $\overline{C_w}=C_w$ and $C_w\equiv T_w$ mod $\Hecke_{>0}$,
where $\Hecke_{>0}:=\sum_{w\in W}q\Alg^+T_w$, see~\cite{Lusztig}. Or more clearly
\[
C_w=T_w+\sum_{y\in W, y<w}p_{y, w}T_y,
\]
where $\leq$ denotes the \textit{Bruhat-Chevalley order} on $W$ and $p_{y, w}\in q\Alg^+$ for all $y<w$ in $W$. We write $y<w$ if $y\leq w$ and $y\neq w$.

The polynomials $p_{y,w}$ are related to the polynomials $P_{y,w}$ of
\cite{KL} (the \emph{Kazhdan-Lusztig polynomials}) by
$p_{y,w}(q)=(-q)^{\ell(w)-\ell(y)}\overline{P_{y,w}(q^2)}$. That is,
to get $p_{y,w}$ from $P_{y,w}$ replace $q$ by $q^2$, apply the bar
involution, and then multiply by $(-q)^{\ell(w)-\ell(y)}$.
\subsection{Multiplication rules for $C$-basis}
For $s\in S, w\in W$, we have
\begin{equation}
T_s C_w=
\begin{cases}
-q^{-1} C_w, \text{if $sw<w$}\\
qC_w+\sum\limits_{y<w, sy<y}\mu(y, w)C_y, \text{if $sw>w$}.
\end{cases}
\end{equation}

 The quantity $\mu(y,w)$, which is the coefficient of
$q^{\frac12(\ell(w)-\ell(y)-1)}$ in $P_{y,w}$, is the coefficient of
$q$ in~$(-1)^{\ell(w)-\ell(y)}p_{y,w}$. However, since Kazhdan and
Lusztig show that $\mu(y,w)$ is nonzero only when $\ell(w)-\ell(y)$
is odd, therefore $\mu(y,w)\in \Z$ can also be described as the coefficient of $q$ in $-p_{y,w}$, as above.

The following notion of $W\!$-graph was introduced by Kazhdan and Lusztig
in~\cite{KL}.
\subsection*{Definition of $W$-graph}
Since we have slightly modified the definition of Hecke algebra
used in \cite{KL}, we are forced to also slightly alter the
definition of $W\!$-graph. We define a \textit{$W\!$-graph
datum\/} to be a triple $(\Gamma, I, \mu)$ consisting of a set
$\Gamma$ (the vertices of the graph), a function
\[
I: \gamma \mapsto I_{\gamma}
\]
from $\Gamma$ to the set of all subsets of $S$, and a function
\[
\mu: \Gamma \times \Gamma \to \Z
\]
such that $\mu(\delta,\gamma) \neq 0$ if and only if $\{\delta,
\gamma\}$ is an edge of the graph. These data are subject to the
requirement that $\Alg\Gamma$, the free $\Alg$-module on~$\Gamma$,
has an $\Hecke$-module structure satisfying
\begin{equation}\label{eq:2}
T_s\gamma=
\begin{cases}
-q^{-1}\gamma& \text{if $s\in I_\gamma$}\\[5pt]
q\gamma+\sum_{\{\delta\in\Gamma\mid s\in I_\delta\}}
\mu(\delta,\gamma)\delta& \text{if $s \notin I_\gamma$,}
\end{cases}
\end{equation}
for all $s\in S$ and $\gamma\in \Gamma$. If $\tau_s$ is the
$\Alg$-endomorphism of $\Alg\Gamma$ such that $\tau_s(\gamma)$ is
the right-hand side of~Eq.~(\ref{eq:2}) then this requirement is
equivalent to the condition that for all $s,\,t\in S$ such that
$st$ has finite order, we require that
\[ \underbrace{\tau_s\tau_t\tau_s\ldots}_{\text{$m$ factors}}
=\underbrace{\tau_t\tau_s\tau_t\ldots}_{\text{$m$ factors}}
\]
where $m$ is the order of~$st$. (Note that the definition of
$\tau_s$ guarantees
that \\
$(\tau_s+q^{-1})(\tau_s-q)=0$ for all $s\in S$.)

For simplicity, if $(\Gamma,I, \mu)$ is a $W\!$-graph datum, we
say that $\Gamma$ is $W\!$-graph. We call $I_\gamma$ the
\textit{descent set\/} of the vertex $\gamma\in\Gamma$, and we
call $\mu(\delta,\gamma)$ and $\mu(\gamma,\delta)$ the
\textit{edge weights\/} associated with the edge
$\{\delta,\gamma\}$. In almost all the cases we consider it turns
out that $\mu(\gamma,\delta)=\mu(\delta,\gamma)$.

\subsection{Cells in $W\!$-graphs}
Following~\cite{KL}, given any $W\!$-graph $\Gamma$ we define a preorder relation
$\leq$ on $\Gamma$ as follows: for $\gamma, \gamma' \in \Gamma$ we say that $\gamma
\leq_{\Gamma} \gamma'$ if there exists a sequence of vertices
$\gamma= \gamma_0,\gamma_1,\cdots \gamma_n=\gamma'$ such that for each $i$
($1\leqslant i\leqslant n$), we have both $\mu(\gamma_{i-1},\gamma_i)\neq 0$ and
$I_{\gamma_{i-1}}\nsubseteq I_{\gamma_i}$. We shall refer to $\leq_\Gamma$
as the \textit{Kazhdan-Lusztig preorder\/} on~$\Gamma$.

Let $\thicksim$ be the equivalence relation on $\Gamma$ associated to the
Kazhdan-Lusztig preorder; thus $\gamma \thicksim \gamma'$ means that
$\gamma \leq_{\Gamma} \gamma'$ and $\gamma' \leq_{\Gamma} \gamma$. The corresponding
equivalence classes are called the \textit{cells\/} of $\Gamma$.

In this paper, the preorder  $\leq_{\Gamma}$ is generated by \textit{Kazhdan-Lusztig left preorder}~\cite{KL}:
$x\leq_\mathcal{L}y$ if $C_x$ occurs with nonzero coefficient in the expression of $T_sC_y$ in the $C$-basis, for some $s\in S$.
Their equivalence classes are called $\emph{left cells}$, see~\cite{KL, Lusztig, Shi} where $\emph{right cells}$ and $\emph{two-sided cells}$ are also defined.
\subsection{Left cell module}
Let $\mathfrak{C}$ be a left cell or, more generally, a union of left cells of $W$. We define an $\Hecke$-module by
$[\mathfrak{C}]_\Alg:=\mathfrak{J}_\mathfrak{C}/\hat{\mathfrak{J}_\mathfrak{C}}$ where
\[
\mathfrak{J}_\mathfrak{C}:=\langle C_w|\text{$w\leqslant_\mathcal{L}z $ for some $z\in\mathfrak{C}$}\rangle_\Alg
\]
\[
\hat{\mathfrak{J}_\mathfrak{C}}:=\langle C_w|\text{$w\notin \mathfrak{C}, w\leqslant_\mathcal{L}z $ for some $z\in\mathfrak{C}$}\rangle_\Alg
\]
are the $\Alg$-spanned modules.

This paper is organized as follows. In Sect. 1 we introduce the indexing sets $D_J, \overline{D_J}$ for the basis of $\Hecke$-module $\Hecke C_{w_J}$, and $E_J$ for the so called \textit{general Specht module}. In Sect. 2, we obtain a version of cellular basis for $\Hecke$ in general and set up the concept of \textit{general Specht module}. In Sect. 3 we show the construction of $W\!$-graph basis by introducing a new family of \textit{$E_J$-Kazhdan-Lusztig polynomials $p_{x y}$}, and show an inductive procedure for computing $p_{x y}' s$. In Sect.4 we consider an example of type $A$ and discuss the applications of our results, we show the transition between Murphy basis and $W\!$-graph basis.
\section{The indexing sets}
For each $J\subseteq S$, let $\hat{J}=S\backslash J$(the
complement of $J$) and define $W_J=\langle J\rangle$, the
corresponding parabolic subgroup of~$W$ and let $w_J\in W_J$ be the unique element of maximal length .  Let $\Hecke_J$ be the
Hecke algebra associated with~$W_J$. As is well known, $\Hecke_J$
can be identified with a subalgebra of~$\Hecke$.

\subsection{Sets $D_J$, $\overline{D_J}$ and $E_J$}
Let $D_J=\{\,w\in W\mid \ell(ws)>\ell(w)\text{ for all $s\in
J$}\,\}$, the set of minimal coset representatives of~$W/W_J$. The
following lemma is well known, it is also an easy consequence of~\cite[Prop. 5.9]{Humphreys}.
\def\deolem{\cite[Lemma 3.2]{D}}
\begin{lemma}[Deodhar \deolem]\label{lemma:4}Let $J\subseteq S$ and
$s\in S$, and define
\begin{align*}
D_{J,s}^-&=\{\,d\in D_J\mid \text{$\ell(sd)<\ell(d)$}\,\},\\
D_{J,s}^+&=\{\,d\in D_J\mid \text{$\ell(sd)>\ell(d)$ and
   $sd\in D_J$}\,\},\\
D_{J,s}^0&=\{\,d\in D_J\mid \text{$\ell(sd)>\ell(d)$ and
   $sd\notin D_J$}\,\},
\end{align*}
so that $D_J$ is the disjoint union $D_{J,s}^-\cup D_{J,s}^+\cup
D_{J,s}^0$. Then $sD_{J,s}^+=D_{J,s}^-$, and if\/ $d\in D_{J,s}^0$
then $sd=dt$ for some $t\in J$.
\end{lemma}
Define
\begin{equation}\label{eq:3}
E_J=\{\,d\in W\mid \text{$\ell(ds)<\ell(d)$ for all $s \in J$ and
$\ell(ds)>\ell(d)$ for all $s \notin J$}\,\}
\end{equation}
that is, $E_J$ is the set of maximal coset representatives of~$W/W_J$
and the minimal ones of~$W/W_{\hat{J}}$. Clearly $\sharp
E_J=\sharp E_{\hat{J}}$, where $E_{\hat{J}}$ was introduced and
written as $Y_J$ in~\cite{Solomon}.

Let $\leq_\mathscr{L}$ denote the \textit{left weak Bruhat order} on $W$.
That is, $x\leq_\mathscr{L}y$ if and only if $y=wx$ for
some $w \in W$ such that $\ell(y)=\ell(w)+\ell(x)$.  McDonough-Pallikaros~\cite{Pall} also say that $x$ is a {\it prefix of $y$} if $x\leq_\mathscr{L}y$. Given $x,y\in
W$ let $[x,y]_\mathscr{L}=\{z \in  w \mid x\leq_\mathscr{L}
z\leq_\mathscr{L} y\}$ be the left interval they determine.

Let
\[
\overline{D_J}=D_Jw_J,
\]
 then
\[
\overline{D_J}=\{\,d\in W\mid \text{$\ell(ds)<\ell(d)$ for all $s
\in J$}\}
\]
is the set of longest coset representatives of $W_J$ in $W$. Thus,
\[ E_J=\overline{D_J}\cap
D_{\hat{J}},
\]
and directly from the definition,
\[\overline{D_J}=\bigcup_{J\subseteq
K\subseteq S}E_K,\] where the union is disjoint.
\begin{proposition}\label{prop:3}Let $J\subseteq S$ and
$s\in S$, we define
\begin{align*}
E_{J,s}^-&=\{\,d\in E_J\mid \text{$\ell(sd)<\ell(d)$ and $sd\in E_J$}\,\},\\
E_{J,s}^+&=\{\,d\in E_J\mid \text{$\ell(sd)>\ell(d)$ and
   $sd\in E_J$}\,\},\\
E_{J,s}^0&=\{\,d\in E_J\mid \text{ $sd\notin E_J$}\,\}
\end{align*}
so that $E_J$ is the disjoint union $E_{J,s}^-\cup E_{J,s}^+\cup
E_{J,s}^0$, then $sE_{J,s}^+=E_{J,s}^-$; let
\begin{align*}
E_{J,s}^{0,-}&=\{\,d\in E_J\mid \text{$\ell(sd)<\ell(d)$ and
   $sd\notin E_J$}\,\},\\
E_{J,s}^{0,+}&=\{\,d\in E_J\mid \text{$\ell(sd)>\ell(d)$ and
   $sd\notin E_J$}\,\},
\end{align*}
then $E_{J,s}^0=E_{J,s}^{0,-}\bigcup E_{J,s}^{0,+}$(disjoint
union); if\/ $d\in E_{J,s}^{0,-}$ then $sd=dt$ for some $t\in J$,
if \/ $d\in E_{J,s}^{0,+}$ then $sd=dt$ for some $t\in \hat{J}$.
\end{proposition}
\begin{proof}
For any $d \in E_J$, we write $d=d'w_J$, where $d'\in D_J$ and
$w_J$ the longest element of $W_J$. Given $s\in S$, we have either
$sd<d$ or $sd>d$.

Case(a): if $sd<d$ then we have either $sd\in E_J$ or $sd\notin
E_J$. If $sd\in E_J$ then $d\in E_{J,s}^-$.

We now consider the case $sd\notin E_J$.
Since $d\in E_J$( that is, $d\in \overline{D_J}$ and $d\in
D_{\hat{J}}$)and $sd<d$, according to Lemma~\ref{lemma:4} we have
$sd \in D_{\hat{J}}$ . Thus $sd\notin \overline{D_J}$, that is
$sd'\notin D_J$, this is the case $d'\in D_{J,s}^0$ in the statement of
Lemma~\ref{lemma:4} , so we have $sd'>d'$ and $sd'=d't$ for some
$t\in J$, and
\[
sd=s(d'w_J)=(sd')w_J=(d't)w_J=(d'w_J)t'=dt'
\]
where $t'=w_J t w_J\in J$. This is the case $d\in E_{J,s}^{0,-}$.

Case(b): if $sd>d$ then again we have either $sd\in E_J$ or $sd\notin
E_J$. If $sd\in E_J$ then $d\in E_{J,s}^+$. we consider the case $sd\notin
E_J$.

Since $sd=s(d'w_J)=(sd')w_J$,  where $d'\in D_{J,s}^+ $
( according to the above discussion, the case
$d' \in D_{J,s}^0$ can not happen , and clearly $d'\notin D_{J,s}^-$). So $sd\in
\overline{D_J}$, and by the assumption $sd \notin E_J$, we
have $sd\notin D_{\hat{J}}$.

Applying Lemma~\ref{lemma:4} to the set $D_{\hat{J}}$, we
have $sd=dt$ for some $t\in \hat{J}$,  which is the case $d\in
E_{J,s}^{0,+}$.
\end{proof}

For $w\in W$ we set $\mathcal{L}(w)=\{s\in S; sw<w\}, \mathcal{R}(w)=\{s\in S; ws<w\}$ and refer them to be the~\emph{left and right descent set} of $w$.

\begin{lemma}~\cite{KL}\cite[Prop.8.6]{Lusztig}
Let $w, w'\in W$, then

(a) if $w\leq_{\mathcal{L}}w'$, then $\mathcal{R}(w')\subseteq \mathcal{R}(w)$. If $w\sim_{\mathcal{L}}w'$, then $\mathcal{R}(w')=\mathcal{R}(w)$.

(b) if $w\leq_{\mathcal{R}}w'$, then $\mathcal{L}(w')\subseteq \mathcal{L}(w)$. If $w\sim_{\mathcal{R}}w'$, then $\mathcal{L}(w')=\mathcal{L}(w)$.
\end{lemma}
The linear map $\varepsilon_J: \Hecke_J\rightarrow\Alg$ defined by $\varepsilon_J(T_w)=\epsilon_w q^{-\ell(w)}$ for any $w\in W_J$ is an algebra
homomorphism, called the sign representation. We denote by $\Ind^S_J(\varepsilon_J)$, the $\Hecke$-module obtained by induction from $\varepsilon_J$.

We now introduce the element $C_{w_J}$ in the Kazhdan-Lusztig $C$-basis of $\Hecke$. By~\cite[Cor. 12.2]{Lusztig},  it has the expression
\[
C_{w_J}=\epsilon_{w_J} q^{\ell(w_J)}\sum_{w\in W_J}\epsilon_w q^{-\ell(w)}T_w.
\]
\begin{lemma}~\cite[Lemma 2.8]{Geck1}
The followings hold

(a) For any $w\in W_J$, we have $T_wC_{w_J}=\epsilon_w q^{-\ell(w)}C_{w_J}$.

(b) We have $C_{w_J}^2=\epsilon_{w_J} q^{-\ell(w_J)}P_JC_{w_J}$, where $P_J=\sum_{w\in W_J}q^{2\ell(w)}$.

(c) The set $\overline{D_J}=D_J w_J$ is a union of left cells in $W$, we have
\[
\overline{D_J}=\{w\in W\mid w\leq_\mathcal{L}w_J\},
\]
and $[\overline{D_J}]_\Alg\cong \Ind^S_J(\varepsilon_J)\cong\Hecke C_{w_J}$ (isomorphisms as left $\Hecke$-modules).
\end{lemma}

\begin{proposition}
For $J\subseteq S$, then

(1) $E_J$ is the left cell, or union of left cells with right descent set $J$.

 (2) The Bruhat order $\leq$ for the elements of $E_J$ is exactly the weak order $\leq_\mathscr{L}$. If
 $x, y\in E_J$ and $x\leq y$, then $\bigl[x, y\bigr]_\mathscr{L} \subseteq E_J$.
\end{proposition}
\begin{proof}
(1) is directly from Lemma 1.3 and 1.4.

(2) is from Prop. 1.2.
\end{proof}

\textbf{Remark} For convenience, in the following sections we still use the usual notations of Bruhat order $\leq , <$ for the weak Bruhat orders $\leq_\mathscr{L}, <_\mathscr{L}$ for the elements of $E_J$, unless indicated.
\subsection{Some multiplication rules}
For $J\subseteq S$, let $M^J=\Hecke C_{w_J}$ be a $\Hecke$-module, then
\begin{lemma}
(1) Let $J\subseteq S$, then $M^J$ is a free $\Alg$-module with basis
\[
\text{$\{ T_wC_{w_J}\mid w\in D_J\}$, or alternatively $\{ T_wC_{w_J}\mid w\in \overline{D_J}\}$}.
\]
the multiplication of $\Hecke$ with respect to this basis:
\[
T_s(T_wC_{w_J})=
\begin{cases}
T_{sw}C_{w_J}+(q-q^{-1})T_w C_{w_J}&\text{if $w\in D_{J,s}^-$ or $w\in \overline{D}_{J,s}^-$}\\
T_{sw}C_{w_J}
&\text{if $w\in D_{J,s}^+$ or $w\in \overline{D}_{J,s}^+$}\\
-q^{-1}T_w C_{w_J} & \text{if $w\in D_{J,s}^0$ or $w\in\overline{D}_{J,s}^0$ }
\end{cases}
\]
for all $s\in S$.

(2) For $w\in E_J$, we have :
\[
T_s(T_wC_{w_J})=
\begin{cases}
T_{sw}C_{w_J}+(q-q^{-1})T_w C_{w_J}&\text{if $w\in E_{J,s}^-$}\\
T_{sw}C_{w_J}
&\text{if $w\in E_{J,s}^+$}\\
-q^{-1}T_w C_{w_J} & \text{if $w\in E_{J,s}^{0,-}$ }\\
qT_wC_{w_J}+T_wC_{tw_J}& \text{if $w\in E_{J,s}^{0,+}, t=w^{-1}sw\in \hat{J}$}
\end{cases}
\]
\end{lemma}
\begin{proof}(1) $M^J$ is spanned by the elements $T_wC_{w_J}$, where $w\in W$; however, if $w=dv$ for $d\in D_J$ and $v\in W_J$, then $T_wC_{w_J}=\varepsilon_v q^{-\ell(v)}T_d C_{w_J}$. It follows that $M^J$ is a free $\Alg$-module with the basis shown and it remains to verify the multiplication formulae.

According to Eq. (1) we immediately get the first two rules. By the multiplication formula for the $C$-basis elements( Eq. (2)), we have:
\[
T_s C_{w_J}=
\begin{cases}
-q^{-1}C_{w_J}&\text{if $s\in J$}\\
qC_{w_J}+ C_{s w_J}&\text{if $s\in \hat{J}$}
\end{cases}
\]
if $w\in D_{J,s}^{0}$, let $t=w^{-1}sw$ and $t\in J$ then $sw=wt<w$, we have
\begin{align*}
T_s(T_wC_{w_J})&=\bigl[T_{sw}+(q-q^{-1})T_w \bigr]C_{w_J}\\
&=\bigl[T_{wt}+(q-q^{-1})T_w \bigr]C_{w_J}\\
&=\bigl[T_{wt}(T_t{T_t}^{-1})+(q-q^{-1})T_w \bigr]C_{w_J}\\
&=T_w{T_t}^{-1}C_{w_J}+(q-q^{-1})T_w C_{w_J}\\
&=T_w\bigl[T_t+(q^{-1}-q)\bigr]C_{w_J}+(q-q^{-1})T_w C_{w_J}\\
&=-q^{-1}T_wC_{w_J}.
\end{align*}
(2)
If $w\in E_{J,s}^{0,+}$ and $t=w^{-1}sw\in \hat{J}$, again by the multiplication rules for $C_{w_J}$
\[
T_s(T_wC_{w_J})=T_w(T_tC_{w_J})=T_w(qC_{w_J}+C_{tw})
\]
\end{proof}
\section{A cellular basis and generic Specht modules}
The concept of "cellular algebras" was introduced by Graham-Lehrer~\cite{Lehrer}. It provides a systematic framework for studying the representation theory of non-semisimple algebras which are deformations of semisimple ones. The original definition was modeled on properties of the Kazhdan-Lusztig basis~\cite{KL} in Hecke algebras of type
$A$. There is now a significant literature on the subject, and many classes of algebras have been shown to admit a "cellular" structure, including Ariki-Koiki algebras, $q$-Schur algebras, Temperly-Lieb algebras, and a variety of other algebras with geometric connections.

As we discussed above, $\Hecke$ is the one-parameter Hecke algebra associated to finite Weyl group $W$.  Furthermore, if $\Hecke$ is defined over a ground ring in which "bad" primes for $W$ are invertible, Geck~\cite{Geck2} used deep properties of the Kazhdan-Lusztig basis and Lusztig's $\mathbf{a}$-function,  he showed that $\Hecke$ has a natural cellular structure in the sense of Graham-Lehrer.

For the purpose of this paper, we show a new version of cellular basis of $\Hecke$. Thus, we also obtain a general theory of "Specht modules" for Hecke algebras of finite type.

We introduce an $\Alg$-linear anti-involution: $\ast: \Hecke\longrightarrow \Hecke$ by $ T_w^\ast=T_{w^{-1}}$ for $w\in W$. Clearly,  $C_{w_J}^\ast=C_{w_J}$; for any $J\subseteq S$ and let $x, y\in D_J$ (or $x, y\in \overline{D_J}$), we define $m_{xy}=T_x C_{w_J}T_y^\ast$. Then $m_{xy}^\ast= m_{yx}$. For convenience, we use the indexing set $\overline{D_J}$ in the following context.

\textbf{Remark} If $J=\emptyset$ then $D_J=W$, as an $\Alg$-modules, $M^\emptyset=\Hecke$ so the elements
\[
\{m_{xy}\mid x, y\in \overline{D_\emptyset}\}
\]
 certainly span $\Hecke$. 

In order to show that $\Hecke$ is cellular, we have to show that $m_{xy}$ with $x, y\in \overline{D_J}$, can be written as an $\Alg$-linear combination of $\{m_{uv}\mid u, v\in E_K, J\subseteq K\}$.
\begin{lemma}
For any $x\in \overline{D_J}$, we have
\[
T_x C_{w_J}=\sum_{x'\in E_J}r_{x'} T_{x'}C_{w_J}+\!\sum_{u\in E_K, J\subsetneq K}r_u T_u C_{w_K}.
\]
where $r_{x'}, r_u\in \Alg$.
\end{lemma}
\begin{proof} As we have found $\overline{D_J}=\bigcup\limits_{J\subseteq
K\subseteq S}E_K$, where the union is disjoint.
If $x \in E_J$ there is nothing to prove; suppose that $x\notin E_J$, then $x\in E_K$ where $K\supsetneq J$. By Prop. 1.2 we have $x=ww_K$ and $w_K=g w_J$
 where $w\in W$(or more exactly $w\in D_K$) and $g\in D_J^K=D_J\cap W_K$, with $\ell(x)=\ell(w) +\ell(w_K)$ and $\ell(w_K)=\ell(g)+\ell(w_J)$.

Since $T_gC_{w_J}$ is the sum of $C_{gw_J}=C_{w_K}$ and a linear combination of terms $C_{hw_J}$ where $h\in D_J^K$ and $h<g$ (this is the special case of~\cite[Prop.2.3]{Geck3}). On the other hand, $C_{hw_J}$ is the sum of $T_h C_{w_J}$ and an $\Alg$-linear combination of terms $T_{f}C_{w_J}$, where $f<h, f\in D_J^K$. As a result, $T_gC_{w_J}$ is the sum of $C_{w_K}$ and an $\Alg$-linear combination of these terms $T_{f}C_{w_J}$. Thus
\begin{align*}
T_x C_{w_J}&=T_{w(gw_J)}C_{w_J}\\
&=\epsilon_{w_J}q^{-\ell(w_J)}T_w(T_g C_{w_J})\\
&=\epsilon_{w_J}q^{-\ell(w_J)}T_w\bigl(C_{w_K}+\sum_{f<g, f\in D_J^K }r_fT_{f}C_{w_J}\bigr)\\
&=r_wT_w C_{w_K} + \sum_{z\in \overline{D_J}, z<wg}r_zT_z C_{w_J}
\end{align*}
where $r_w, r_f, r_z \in \Alg$. By induction, each term $T_z C_{w_J}$ has also the required form.
\end{proof}


\begin{lemma}
Let $J\subseteq S$ and suppose that $x, y\in \overline{D_J}$, then there exist $r_{x'y}, r_{uv}\in \Alg$ such that
\[
m_{xy}=\sum_{x'\in E_J}r_{x'y} m_{x'y}+\!\sum_{u\in E_K, v\in \overline{D_K},J\subsetneq K}r_{uv} m_{uv}.
\]
\end{lemma}
\begin{proof}
By Lemma 2.1, we have
\begin{align*}
m_{xy}&=T_xC_{w_J}T_y^\ast\\
&=\bigl[\sum_{x'\in E_J}r_{x'} T_{x'}C_{w_J}+\!\sum_{u\in E_K, J\subsetneq K}r_u T_u C_{w_K}\bigr]T_y^\ast\\
&=\sum_{x'\in E_J}r_{x'} T_{x'}C_{w_J}T_y^\ast+\!\sum_{u\in E_K, J\subsetneq K}r_u T_u C_{w_K}T_y^\ast
\end{align*}
and
\[
C_{w_K}T_y^\ast=(T_y C_{w_K})^\ast
\]
where $T_y C_{w_K}\in \Hecke C_{w_K}$, this implies $T_y C_{w_K}\in \langle T_v C_{w_K}\mid v\in \overline{D_K}\rangle_\Alg$, as required.
\end{proof}
Let $\Omega^{lex}=\{J\mid J\subseteq S\}$ be a set ordered lexicographically.
\begin{thm}
The Hecke algebra $\Hecke$ is free as an $\Alg$-module with basis
\[
\text{$\mathcal{M}=\{m_{uv}\mid u, v\in E_J$ for some $J\subseteq S \}$}.
\]
\end{thm}

\begin{proof}
 We first show that $\mathcal{M}$ spans $\Hecke$ by showing that whenever $x, y\in \overline{D_J}$ then $m_{xy}$ can be written as a $\Alg$-linear combination of terms $m_{uv}$ in $\mathcal{M}$. When $J=S$ this is clear because $\Hecke C_{w_J}\Hecke=\Alg C_{w_J}$. If $J\neq S$, by Lemma 2.2 , we have
\[
m_{xy}=\sum_{x'\in E_J}r_{x'y} m_{x'y}+\!\sum_{(u,v),J\subsetneq K}r_{uv} m_{uv},
\]
where $ r_{x'},  r_{uv}\in \Alg$, and the second sum is over the pairs $(u, v)$ where $u\in E_K$, $v\in \overline{D_K}$. However, $m_{xy}^\ast=m_{yx}$ so by induction on the elements of $\Omega^{lex}$ again
( start with $J=S$,  clearly $C_{w_J}^\ast=C_{w_J}$), $m_{xy}$ can be written as an $\Alg$-linear combination of elements of $\mathcal{M}$. Finally, let $J=\emptyset$, then $\Hecke=\Hecke C_{w_\emptyset}\Hecke$.

Therefore $\mathcal{M}$ spans $\Hecke$ .

By Wedderburn's theorem $\textit{dim}(\Hecke)=|W|=\sum\limits_{J\subseteq S}|\mathcal{M}(J)|^2$, where
\[
\text{$\mathcal{M}(J)=\{m_{uv}\mid u, v\in E_J$ for a fixed $J, J\subseteq S \}$}.
\]
Hence the set $\mathcal{M}$ has the correct cardinality.
\end{proof}
Define $\hat{\Hecke^J}$ to be the $\Alg$-module with basis
 \[
 \text{$\{ m_{uv}\mid w, v\in E_K$ for some $K$ such that $J\subset K\subseteq S$\}}.
 \]
 where we write $J\subset K$ when $J\subseteq K$ and $J\neq K$.
Similarly, we define $\Hecke^J$ to be the $\Hecke$-module with basis $m_{uv}$ where $u, v\in E_K$ with $J\subseteq K\subseteq S$.
\begin{thm}
(1) The $\Alg$-linear map determined by
\[
m_{uv}\longmapsto m_{vu}
\]
 for all $m_{uv}\in\mathcal{M}$, is an anti-isomorphism of $\Hecke$.

(2) Suppose that $h\in \Hecke$ and that $u\in E_J$, there exist $r_u\in \Alg$ such that for all $v\in E_J$
\[
\text{$hm_{uv}\equiv\sum_{w\in E_J}r_w m_{wv}\!$\ \ \  mod $\hat{\Hecke^J}$}.
\]
Consequently, $\{\mathcal{M}, \Omega^{lex}\}$ is a cellular basis of $\Hecke$.
 \end{thm}
\begin{proof}
(1) The $\ast$-endomorphism and the $\Alg$-linear map determined by $m_{uv}\longmapsto m_{vu}$ coincide since $m_{uv}^\ast=m_{vu}$ for all $m_{uv}$ in $\mathcal{M}$. This proves (1) since $\ast$ is an anti-isomorphism of $\Hecke$

(2) We argue by induction on $J\in \Omega^{lex}$. By (1), if $J=S$ then $\Hecke C_{w_J}\Hecke=\Alg C_{w_J}$, there is nothing to prove. Suppose that $J\subseteq S$. First we consider $v=w_J$. Since $\mathcal{M}$ is a basis of $\Hecke$, for any $h\in \Hecke$ we may write
\[
h m_{u, w_J}=\sum_{x, y\in E_K,\\ K\subseteq S}r_{xy}m_{xy}
\]
for some $r_{xy}\in \Alg$. Now $h m_{u, w_J}$ belongs to $M^J$, clearly, if $r_{xy}\neq 0$ then $J\subseteq K$; further, if $J=K$ then we must also have $v=w_J$. Hence,
\begin{equation}\label{eq:a}
\text{$h m_{u, w_J}=\sum_{x\in E_J}r_{x}m_{x, w_J}$ \ \ \ mod $\hat{\Hecke^J}$}
\end{equation}
where $r_x=r_{x, w_J}\in \Alg$. This completes the proof of (2) when $v=w_J$.

Now, if $K\supsetneq J$ and $u, y\in E_K$ then $m_{uy} T_v^\ast=(T_v m_{yu})^\ast\in \Hecke^K \subseteq \hat{\Hecke^J}$ by induction on $J\in \Omega^{lex}$. Therefore, we can multiply the Eq.~(\ref{eq:a}) on the right by $T_v^\ast$, to complete the proof.
\end{proof}
So we can now introduce the following:
\begin{defn}
 Let $S^J=\langle T_u C_{w_J} + \hat{\Hecke^J}\mid u\in E_J\rangle_\Alg$, then $S^J$ is an $\Hecke$-submodule of $\Hecke^J/\hat{\Hecke^J}$. We call this the \emph{generic Specht module} of $\Hecke$ associated with $J$.
\end{defn}
\subsection*{The bar involution for $S^J$} 
For all $x,\,y \in E_J$ we define elements $R_{x,y}\in\Alg$ by the
formula
\begin{equation}\label{eq:6}
\overline{T_yC_{w_J}}
   =\!\!\ \ \sum_{x \in E_J}
     R_{x,y} T_x C_{w_J}\ \ \  \text{mod $\hat{\Hecke}^J$},
\end{equation}

 We can easily derive the following formulae which provide an inductive procedure for
calculating these elements in $S^J$.
\begin{proposition}\label{prop:3}
Let $x,\,y\in E_J$. If $s\in S$ is such that $\ell(sy)<\ell(y)$
then
\[
\postdisplaypenalty=10000 \advance\abovedisplayskip 0 pt minus 3
pt \advance\belowdisplayskip 0 pt minus 3 pt R_{x,y}(\text{mod $\hat{\Hecke}^J$})=
\begin{cases}
R_{sx,sy}&\text{if $x\in E_{J,s}^-$}\\
R_{sx,sy}+(q^{-1}-q)R_{x,sy}
&\text{if $x\in E_{J,s}^+$}\\
-qR_{x,sy} & \text{if $x\in E_{J,s}^{0,-}$ }\\
q^{-1}R_{x,sy}& \text{if $x\in E_{J,s}^{0,+}$}
\end{cases}
\]
\end{proposition}
We may use induction on $\ell(y)$ to establish that $R_{x,y}=0$
unless $x\leqslant_\mathscr{L} y$ in the weak Bruhat partial order on~$E_J$; this
follows from the fact that if $sy\leqslant_\mathscr{L} y$ and $x\leqslant_\mathscr{L} sy$
then both $x\leqslant_\mathscr{L} y$ and $sx\leqslant_\mathscr{L} y$. It is also easily seen that $R_{x,x}=1$.
\section{$W\!$-graphs for generic Specht modules}\label{construction}
Let $\mathfrak{C}_{w_J}$ be a left cell, or more generally, a union of left cells containing $w_J$, then the transition between the bases of
 the left cell module $[\mathfrak{C}_{w_J}]_\Alg$ and the generic Specht module $S^J$ is described as the following:
\begin{thm}\label{thm:2}
The $\Hecke$-module $S^J$ has a unique basis
$\{\,C_w\mid w\in E_J\,\}$ such that $\overline{C_w}=C_w$ for all
$w\in E_J$, and
\[
C_w\,=\!\ \ \sum\limits_{y \in E_J}
  \!\!P_{y,w}T_yC_{w_J}\!\ \ \text{mod $\hat{\Hecke^J}$}
\]
for some elements $P_{y,w}\in\Alg^+$ with the following
properties\/\textup{:}
\begin{itemize}
\item[(i)] $P_{y,w}=0$ if $y\nleqslant w$\textup{;} \item[(ii)]
$P_{w,w}=1$; \item[(iii)] $P_{y,w}$ has zero constant term if
$y\neq w$.
\end{itemize}
\end{thm}
Comparing with the original Kazhdan-Lusztig's polynomials
in~\cite{KL}, we called $\{P_{y,w}\mid y, w \in E_J\}$ the family
of \textit{$E_J$-relative Kazhdan-Lusztig polynomials}. We
shall show that the basis $\{C_w\mid w\in E_J\}$ give
$S^J$ the structure of a $W\!$-graph. That
is, there is a $W\!$-graph $\Lambda$ with vertex elements
$\{C_w\mid w \in E_J\}$.
Before showing the proof of Theorem~\ref{thm:2}, we describe the
edge weights and descent sets for~$\Lambda$.

Given $y,\,w\in E_J$ with $y \ne w$,  we define an integer
$\mu(y,w)$ as follows. If $y<w$ then $\mu(y,w)$ is the coefficient
of $q$ in $-P_{y,w}$.

We write $y \prec w$ if $y<w$  and $\mu(y,w)\ne 0$.

The \textbf{(left) descent set} associated with the vertex element
$C_w (w \in E_J)$ of $\Lambda$ is
\begin{align*}
I(w)&=\{\,s\in S\mid\ell(sw)<\ell(w)\}\\
& = \{\,s\in S\mid w \in E_{J,s}^-\} \cup\{s \mid w\in
E_{J,s}^{0,-} \}
\end{align*}

In accordance with the notation introduced in Section~2, we define
\begin{align*}
\Lambda_s^-
&=\{\,w\in E_J\mid s\in I(w)\,\}\\
&=\{\,w \mid \text{$w\in E_{J,s}^-$ or $w\in E_{J,s}^{0,-}$}\},
\end{align*}
and similarly $\Lambda_s^+=\{\,w\in E_J\mid s\notin I(w)\,\}$. Our
proof of Theorem~\ref{thm:2} will also incorporate a proof of the
following result, which will be an important component of the
subsequent proof that $\Lambda$ is a $W\!$-graph.

\begin{thm}
\label{thm:3} Let $v\in E_J$. Then for all $s\in S$ such that
$\ell(sv)>\ell(v)$ and $sv\in E_J$ we have
\[
T_sC_v=qC_v+C_{sv} +\sum_{z\in E_J} \mu(z,v)C_z,
\]
where the sum is over all $z\in\Lambda_s^-$ such that $z\prec v$.
\end{thm}
The following is the proof of Theorem 4.1.
\begin{proof}
Uniqueness is proved similarly with that of \cite[Theorem
1.1]{KL}, we omit the details.

Existence. We give a recursive procedure for constructing elements
$P_{x,w}$ satisfying the requirements of Theorem~\ref{thm:2}. We
start with the definition
\[
 P_{w_J,w_J}=1
\]
so that $\overline{C_{w}}=C_{w}$ holds for~$w=w_J$, as do
Conditions~(i), (ii) and~(iii).

Now assume that $w\ne w_J$ and that for all $v\in E_J$ with
$\ell(v)<\ell(w)$ the elements $P_{x,v}$ have been defined (for
all $x\in E_J$ ) so that the requirements of Theorem~\ref{thm:2}
are satisfied. Thus the elements $C_{v}$ are known when
$\ell(v)<\ell(w)$. We may choose $s\in S$ such that $w=sv$ with
$\ell(w)=\ell(v)+1$; note that $v\in E_J$ by Lemma~1.6. In
accordance with the formula in Theorem~\ref{thm:3} we define
\begin{equation}\label{eq:80}
C_{w}=(T_s-q)C_{v}-\!\!
\sum_{\substack{z\prec v\\
z\in\Lambda_s^-}} \!\!\mu(z,v)C_{z}.
\end{equation}
Since $\overline{T_s-q}=T_s-q$, induction immediately gives
$\overline{C_w}=C_w$. We define $P'_{y,w}$ and $P''_{y,w}$ by
\begin{align}
(T_s-q)C_v &=\!\ \ \sum_{y \in E_J }
  P'_{y,w}T_yC_{w_J}\label{eq:5}\\
\sum_{z\prec v} \!\!\mu(z,v)C_{z} &=\!\ \ \sum_{y \in E_J}
  P''_{y,w}T_yC_{w_J}\label{eq:6}
\end{align}
and define $P_{y,w}=P'_{y,w} -P''_{y,w}$.

If $y\in E_J$ then
\[
(T_s-q)T_y=
\begin{cases}
T_{sy}-qT_y&\text{if $y\in E_{J,s}^+$}\\
T_{sy}-q^{-1}T_y&\text{if $y\in E_{J,s}^-$}\\
T_y(T_t-q)&\text{if $y\in E_{J,s}^{0,-}$}\\
T_{sy}-qT_y&\text{if $y\in E_{J,s}^{0,+}$}\\
\end{cases}
\]
where we have written $t=y^{-1}sy$ in the case $y\in E_{J,s}^0$.
Thus we see that
\begin{align*}
(T_s-q)C_v & =\!\!\sum_{y\in
E_{J,s}^+}\!\!
  P_{y,v}(T_{sy}-qT_y)C_{w_J}\,\,
  +\!\!\sum_{y\in E_{J,s}^-}\!\!
  P_{y,v}(T_{sy}-q^{-1}T_y)C_{w_J}\\
&\qquad +\!\!\sum_{y\in E_{J,s}^{0,-}}\!\!
  {P_{y,v}}T_y(T_t-q)C_{w_J}
  +\!\!\sum_{y\in E_{J,s}^{0,+}}{P_{y,v}}(T_{sy}-qT_y)C_{w_J}\!\!\displaybreak[0]\\
&\kern-18 pt =\!\!\sum_{y\in E_{J,s}^-}\!\!
 (P_{sy,v}-q^{-1}P_{y,v})T_yC_{w_J}
 \,\,+\!\!\sum_{y\in E_{J,s}^+}\!\!
 (P_{sy,v}-qP_{y,v})T_yC_{w_J}\\
&\qquad +\!\!\sum_{y\in E_{J,s}^{0,-}}\!\!
  {P_{y,v}}(-q^{-1}-q)T_yC_{w_J}\\
&\hspace{.2 cm} +\!\sum_{y\in E_{J,s}^{0,+}}{P_{y,v}}
  \Bigl[(qT_yC_{w_J}+T_yC_{tw_J})-qT_yC_{w_J}\Bigr]
\end{align*}

Now comparing Eq.~(\ref{eq:5}) with the expression for
$(T_s-q)C_v$ obtained above we obtain the following formulas for
the cases $y\in E_{J,s}^+$ (case~(a)), $y\in E_{J,s}^-$ (case
(b)), $y\in E_{J,s}^{0,-}$ and (case (c)) and $y\in E_{J,s}^{0,+}$
(case (d)):
\begin{equation}\label{eq:7}
P'_{y,w}=
\begin{cases}
P_{sy,v}-qP_{y,v}
  &\text{(case (a)),}\\[3 pt]
P_{sy,v}-q^{-1}P_{y,v}
  &\text{(case (b)),}\\[3 pt]
(-q-q^{-1})P_{y,v}
&\text{(case (c)),}\\[3 pt]
0&\text{(case (d)).}
\end{cases}
\end{equation}

Since $C_z=\!\ \ \sum_{y\in
E_J}P_{y,z}T_yC_{w_J}$, we have
\[
\sum_{z\prec v,z\in\Lambda_s^-}\!\!\!\mu(z,v)C_z =\!\ \
\sum_{y\in E_J} \sum_{z\prec
v,z\in\Lambda_s^-}\!\!\!\mu(z,v)
  P_{y,z}T_yC_{w_J}
\]
and by comparison with Eq.~(\ref{eq:6})
\begin{equation}\label{eq:8}
P''_{y,w}\,=\!\! \sum_{\substack{z\prec v\\z\in\Lambda_s^-}}
\!\!\mu(z,v)P_{y,z}.
\end{equation}
We may check that with $P'_{y,w}$ and $P''_{y,w}$ given by
Eq's (10)~and~(11), the elements $P_{y,w}=P'_{y,w}-P''_{y,w}$ lie in
$\Alg^+$ and satisfy Conditions (i), (ii) and (iii) of
Theorem~\ref{thm:2}. We omit the details here.
\end{proof}

For convenience, let $\tilde{T_w}=T_wC_{w_J}$. Observe that the formula for $C_w$ in Theorem~\ref{thm:2} may be
written as
\[
C_w=\tilde{T_w}
+\!\!\!\sum_{y<w,y\in E_J}\!\!\!P_{y,w}\tilde{T_y},
\]
and inverting this gives
\begin{equation}\label{eq:10}
\tilde{T_w}=\!\ \ C_w
+\!\!\!\sum_{y<w,y\in E_J}\!\!\!Q_{y,w} C_y
\end{equation}
where the elements $Q_{y,w}$ (defined whenever $y<w$) are given
recursively by
\[
Q_{y,w}=-P_{y,w} -\!\!\!\!\sum_{\{z|y<z<w\}}\!\!\!\!
Q_{y,z}P_{z,w}.
\]
In particular, $Q_{y,w}$ is in $\Alg^+$, has zero constant term,
and has coefficient of $q$ equal to $\mu(y,w)$.

We now state our main result.

\begin{thm}\label{thm:4}
The basis $\{C_w\mid w\in E_J\}$ gives the generic Specht module $S^J$
the structure of a $W\!$-graph, as described above.
\end{thm}

\begin{proof}The proof is similar with~\cite[Theorem 2.6]{Yin},
modified appropriately. We start by using induction on $\ell(w)$
to prove that for all $s\in S$
\begin{equation}\label{eq:11}
T_sC_w=
\begin{cases}
-q^{-1}C_w& \text{if $w\in\Lambda_s^-$},\\[5 pt]
qC_w+\!\!\sum\limits_{z\in E_J, z\in\Lambda_s^-}\!\! \mu(z,w)C_z&
\text{if $w\notin\Lambda_s^-$.}\\[-5 pt]
\end{cases}
\advance\belowdisplayskip5 pt
\end{equation}
or more exactly
\begin{equation}\label{eq:12}
T_sC_w(\text{ mod $\hat{\Hecke^J}$})=
\begin{cases}
-q^{-1}C_w& \text{if $w\in E_{J,s}^{-}$ or $w\in E_{J,s}^{0,-}$},\\[5 pt]
qC_w+ C_{sw}+\!\!\sum\limits_{z\in E_{J,s}^-, z<w}\!\! \mu(z,w)C_z&
\text{if $w\in E_{J,s}^+$.}\\[-5 pt]
qC_w +\!\!\sum\limits_{z\in E_{J,s}^-, z<w}\!\! \mu(z,w)C_z&
\text{if $w\in E_{J,s}^{0,+}$.}\\[-5 pt]
\end{cases}
\advance\belowdisplayskip5 pt
\end{equation}

If $w\in E_{J,s}^+$ then $w\notin\Lambda_s^-$, and
Eq.~(\ref{eq:11}) follows immediately from Theorem~\ref{thm:3}
(applied with $v$ replaced by~$w$), since the only
$z\in\Lambda_s^-$ with $\mu(z,w)\ne0$ and
$\ell(z)\geqslant\ell(w)$ is $z=sw$.

For the case $w\in E_{J,s}^{0,+}$, the term $C_{sw}$
can not appear in the sum of Eq.~(\ref{eq:11}).

If $w\in E_{J,s}^-$, which implies that $w\in\Lambda_s^-$, then
writing $v=sw$ and applying Theorem~\ref{thm:3} gives
\[
C_w=(T_s-q)C_v- \sum\mu(z,v)C_z,
\]
where $z\prec v$ and $z\in\Lambda_s^-$ for all terms in the sum.
The inductive hypothesis thus gives $T_sC_z=-q^{-1}C_z$, and since
we also have $T_s(T_s-q)=-q^{-1}(T_s-q)$ it follows that
$T_sC_w=-q^{-1}C_w$, as required.

Now suppose that $w\in E_{J,s}^0$, and as usual let us write
$sw=wt$. Suppose first that $t=w^{-1}sw\in J$, so that
$w\in\Lambda_s^-$. By Eq.~(\ref{eq:10}),
\[
C_w=\tilde{T_w}-\!\!\!\sum_{\{y|y<w,y\in
E_J\}}\!\!\! Q_{y,w}C_y,
\]
and since
$T_sT_wC_{w_J}+q^{-1}T_wC_{w_J}=T_w(T_tC_{w_J}+q^{-1}C_{w_J})=0$
we find that
\begin{equation}\label{eq:12}
T_sC_w+q^{-1}C_w= -\!\!\!\sum_{\{y|y<w,y\in E_J\}}\!\!\!Q_{y,w}
(T_sC_y+q^{-1}C_y).
\end{equation}
By the inductive hypothesis,
\[
T_sC_y+q^{-1}C_y=
\begin{cases}
0&\text{if $y\in\Lambda_s^-$}\\[5 pt]
(q+q^{-1})C_y+ \!\!\!\!\!\sum\limits_{z\in\Lambda_s^-}
\!\!\!\!\mu(z,y)C_z&\text{if
$y\notin\Lambda_s^-$,}\\[-5 pt]
\end{cases}
\advance\belowdisplayskip 5 pt
\]
and so Eq.~(\ref{eq:12}) gives
\begin{equation}\label{eq:13}
T_sC_w+q^{-1}C_w=
-\!\!\sum_{\substack{y\notin\Lambda_s^-\\y<w}}\!\!
Q_{y,w}(q+q^{-1})C_y\,+\,X
\end{equation}
for some $X$ in the $\Alg$-submodule spanned by the elements $C_z$
for $z\in\Lambda_s^-$. Now since $T_s=T_s^{-1}+(q-q^{-1})$ it
follows that
\begin{align*}
(T_s+q^{-1})C_w&=\overline{(T_s+q^{-1})C_w}\\
&=-\!\!\sum_{\substack{y\notin\Lambda_s^-\\y<w}}\!\!
\overline{Q_{y,w}}(q^{-1}+q)C_y \,+\,\overline X,
\end{align*}
and comparing with Eq.~(\ref{eq:13}) shows that for all $y$ with
$y<w (y\in E_J)$ and $y\notin\Lambda_s^-$,
\begin{equation}\label{eq:14}
\overline{Q_{y,w}}=Q_{y,w}.
\end{equation}
Since $Q_{y,w}$ is in $\Alg^+$ and has zero constant term,
Eq.~(\ref{eq:14}) forces $Q_{y,w}$ to be zero whenever $y<w$ and
$y\notin\Lambda_s^-$. Therefore the right hand side of
Eq.~(\ref{eq:12}) is zero, since $T_sC_y+C_y=0$ whenever
$y\in\Lambda_s^-$. So
\[
T_sC_w=-q^{-1}C_w,
\]
as required.
\end{proof}.
\section{Applications to type $A$}
Throughout this section, we apply our results to the Hecke algebra of type $A$. Let $W=\mathfrak{G}_n$ be the symmetric group acting on the left on $\{1, 2, \cdots, n\}$.  Another reference is the exposition by Mathas~\cite{Mathas2}. For $i=1,
2,\cdots, n-1$ let $s_i$ be the basic transposition $(i, i+1)$ and
let $S=\{s_1, s_2, \cdots, s_{n-1}\}$, the generating set of $\mathfrak{G}_n$.
\subsection{Notations}
Let $\lambda=(\lambda_1, \lambda_2, \cdots, \lambda_r )$ be a
partition of $n$ with the notation $\lambda\vdash n$.  A standard $\lambda$-tableau is a tableau whose entries are
exactly $1, 2, \cdots, n$ and which has both increasing rows and
increasing columns, the set is denoted $\mathbb{T}(\lambda)$. Let $t^\lambda$ (resp. $t_\lambda$) be the $\lambda$-tableau in
which the numbers $1, 2, \cdots$  appear in order from left to right
(resp. top to bottom) and down along successive rows (resp.
columns), then $t^\lambda, t_\lambda\in \mathbb{T}(\lambda)$. For a Young
tableau $t$, we put
\[
I(t)=\{i\mid 1\leq i\leq n-1, \text{$i+1$ is in a lower position than
$i$ in $t$} \}
\]
and call it the {\it descent set} of $t$. Let
\begin{align*}
I_0(t)&=\{i\in I(t)\mid \text{$i+1$ is in the left side of $i$ in $t$} \},\\
I_1(t)&=\{i\in I(t)\mid \text{$i+1$ is directly below $i$ in $t$} \}.
\end{align*}
\begin{lemma}~\cite{Naruse}
For a standard tableau $t$ of shape $\lambda\vdash n$,
\begin{align*}
 (1) I(t)&=I_0(t)\cup I_1(t);\\
 (2) {I(t)\cup I(t')}&=\{1,2,\dots, n-1\};\\
 (3) I_0(t)&=\emptyset \text{ if and only if $t=t_\lambda$};\\
 (4) I_0(t')&=\text{$\emptyset$ if and only if $t=t^\lambda$}.
 \end{align*}
 \end{lemma}
The Young subgroup $\mathfrak{G}_{\lambda}=\mathfrak{G}_{\lambda_1} \times\dots \times \mathfrak{G}_{\lambda_r}$ of $\mathfrak{G}_n$ is the row stabilizer of $t^\lambda$.
 Let $D_\lambda$ be the set of distinguished left coset representatives of $\mathfrak{G}_{\lambda}$ in $\mathfrak{G}_n$, by Dipper-James~\cite{James} and Mathas~\cite{Mathas2}, we have the following explicit description:
 \[
 D_\lambda=\{w\in \mathfrak{G}_n\mid \text{$wt^\lambda$ is row-standard}\}.
 \]
 As in~\cite{James, Murphy1, Mathas2}, if $t$ is a row-standard $\lambda$-tableau, the unique element $d\in D_\lambda$ such that $t=dt^\lambda$ will be denoted by $d(t)$. Let $w_{J(\lambda)}$ be the longest element of the Young subgroup $\mathfrak{G}_{\lambda}$, an element $w_{\lambda}$ is defined by $t_\lambda =w_{\lambda}t^\lambda$.

 Given partitions $\mu=(\mu_1, \mu_2,...)$ and $\lambda=(\lambda_1, \lambda_2, ...)$ of $n$, we say $\mu$ \emph{dominates} $\lambda$, and write $\lambda\trianglelefteq \mu$, if
 \[
 \lambda_1\leq \mu_1, \lambda_1+\lambda_2\leq \mu_1+\mu_2, \lambda_1+\lambda_2+\lambda_3\leq \mu_1+\mu_2+\mu_3, ...
 \]
 we write $\lambda\trianglelefteq \mu$ if $\lambda\trianglelefteq \mu$ and $\mu\neq \lambda$. The partial order $\trianglelefteq$ on the set of partitions(or shapes) of $n$ will be referred to as the \emph{dominance order}.

 For a fixed $\lambda\vdash n$, $s, t\in \mathbb{T}(\lambda)$. We write $s\trianglelefteq t$ if $\ell(d(s))\leqslant \ell(d(t))$, and $s\vartriangleleft t$ if $s\trianglelefteq t$ and $s\neq t$. We note that the notation here is different with~\cite{Mathas2}[pp.31].
\subsection{Cells}
 The cells of $W=\mathfrak{G}_n$ may be described in terms of the Robinson-Schensted correspondence. The correspondence is a bijection of $S_n$ to pairs of standard tableaux $(P, Q)$ of the same shape corresponding to partitions of $n$, so that if $w\longmapsto (P(w), Q(w))$ then $Q(w)=P(w^{-1})$. In particular, the involutions are the elements $w\in W$ for which $Q(w)=P(w)$. If $\lambda \vdash n$, the pair of tableaux corresponding to $w_{J(\lambda)}$ has the form $(t_{\lambda'}, t_{\lambda'})$. Hence, the tableaux corresponding to $w_{J(\lambda)}$ have shape $\lambda'$, where $\lambda'$ denotes the partition conjugate to $\lambda$.

 If $R$ is a fixed standard tableau then the set $\{w\in W: Q(w)=R\}$ is a left cell of $W$ and the set $\{w\in W: P(w)=R\}$ is a right cell of $W$. See~\cite{KL} and also~\cite{Ariki} for an alternative proof of this result.
 \begin{lemma}
 Let $\lambda\vdash n$ and $t\in \mathbb{T}(\lambda)$. The element of $\mathfrak{G}_n$, which corresponds to the pair of tableaux $(t^{\lambda'}, t_{\lambda'})$ under the Robinson-Schensted correspondence, is $w_{\lambda}w_{J(\lambda)}$.
\end{lemma}
The following is the corollaries of the discussion in Section 1, see also in~\cite[Lemma 3.3]{Pall} and Du~\cite[Lemma 1.2]{Du}.
\begin{lemma}
 The followings hold(i) $w_{\lambda} w_{J(\lambda)}\in D_\lambda$, (ii) $dw_{J(\lambda)}\in D_\lambda$ for each prefix $d$ of $w_{\lambda}$, (iii) $dw_{J(\lambda)}\in D_\lambda$ is in the same left cell as $w_{J(\lambda)}$
 for each prefix $d$ of $w_{\lambda}$.
 \end{lemma}
 As in Section 1, we write \text{$E_{J(\lambda)}=\{e\mid e=dw_{J(\lambda)}$ and $d$ is a prefix of  $w_{\lambda} \}$}, for any $s_i=(i, i+1)\in S$ we define
 \begin{align*}
E_{J(\lambda),s_i}^-&=\{\,e\in E_{J(\lambda)}\mid \text{$\ell(s_ie)<\ell(e)$ and $s_ie\in E_{J(\lambda)}$}\,\},\\
E_{J(\lambda),s_i}^+&=\{\,e\in E_{J(\lambda)}\mid \text{$\ell(s_ie)>\ell(e)$ and
   $s_ie\in E_{J(\lambda)}$}\,\},\\
E_{J(\lambda),s_i}^0&=\{\,e\in E_{J(\lambda)}\mid \text{ $s_ie\notin E_{J(\lambda)}$}\,\}
\end{align*}
so that $E_{J(\lambda)}$ is the disjoint union $E_{J(\lambda),s_i}^-\cup E_{J(\lambda),s_i}^+\cup
E_{J(\lambda),s_i}^0$, then
\[s_iE_{J(\lambda),s_i}^+=E_{J(\lambda),s_i}^-;
\]
 let
\begin{align*}
E_{J(\lambda),s_i}^{0,-}&=\{\,e\in E_{J(\lambda)}\mid \text{$\ell(s_ie)<\ell(e)$ and
   $s_ie\notin E_{J(\lambda)}$}\,\},\\
E_{J(\lambda),s_i}^{0,+}&=\{\,e\in E_{J(\lambda)}\mid \text{$\ell(s_ie)>\ell(e)$ and
   $s_ie\notin E_{J(\lambda)}$}\,\},
\end{align*}
then $E_{J(\lambda), s_i}^0=E_{J(\lambda), s_i}^{0,-}\bigcup E_{J(\lambda), s_i}^{0,+}$(disjoint
union); if\/ $e\in E_{J(\lambda), s_i}^{0,-}$ then $s_ie=et$ for some $t\in J(\lambda)$,
if \/ $e\in E_{J(\lambda), s_i}^{0,+}$ then $s_ie=et$ for some $t\in \hat{J(\lambda)}$, where $\hat{J(\lambda)}=S\diagdown J(\lambda)$.

 We have the following observation
  \begin{align*}
E_{J(\lambda),s_i}^-&=\{\,d(t)w_{J(\lambda)}\mid t\in \mathbb{T}(\lambda), i\in I_0(t')\},\\
E_{J(\lambda),s_i}^+&=\{\,d(t)w_{J(\lambda)}\mid t\in \mathbb{T}(\lambda), i\in I_0(t)\},\\
E_{J(\lambda),s_i}^{0,-}&=\{\,d(t)w_{J(\lambda)}\mid t\in \mathbb{T}(\lambda), i\in I_1(t')\},\\
E_{J(\lambda),s_i}^{0,+}&=\{\,d(t)w_{J(\lambda)}\mid t\in \mathbb{T}(\lambda), i\in I_1(t)\},
\end{align*}

 Let
 \[
  C_{w_{J(\lambda)}}=\epsilon_{w_{J(\lambda)}} q^{\ell(w_{J(\lambda)})}\sum_{w\in \mathfrak{G}_{\lambda}}\epsilon_w q^{-\ell(w)}T_w.
 \]
 then the following statement is a corollary of Lemma 2.1.
 \begin{lemma}~\cite{James}{Mathas2}
 Let $\lambda\vdash n$, then $\Hecke C_{w_J(\lambda)}$ is a free $\Alg$-module with basis
 \[
 \{T_{d(t)}C_{w_{J(\lambda)}}|\text{$t$ a row standard $\lambda$-tableau}\}.
 \]
  Moreover, if $t$ is row standard and $s=s_i t$ for some $1\leq i\leq n-1$, then
 \[
 T_iT_{d(t)}C_{w_{J(\lambda)}}=\begin{cases}
 T_{d(s)}C_{w_{J(\lambda)}}, \text{if $i\in I_0(t)$} \\
 T_{d(s)}C_{w_{J(\lambda)}}+(q-q^{-1})T_{d(t)}C_{w_{J(\lambda)}}, \text{if $i\in I_0(t')$} \\
 -q^{-1}T_{d(t)}C_{w_{J(\lambda)}}, \text{if $i\in I_1(t')$}
 \end{cases}
 \] where $T_i:=T_{s_i}$.
 \end{lemma}
 \subsection{Murphy basis and $W\!$-graph basis}
The following is a corollary of the main Theorems in Section 2.
 \begin{thm}~\cite{Murphy1, Murphy2}
 For any $\lambda\vdash n$ and $s, t\in \mathbb{T}(\lambda)$, we define elements of $\Hecke$ by
 \[
 m_{st}=T_{d(s)}C_{w_{J(\lambda)}}T_{d(t)^{-1}}
 \]
 then the following hold
 (a) The set $\{m_{st}| \text{$s, t\in \mathbb{T}(\lambda)$ for some $\lambda\vdash n$}\}$ is an $\Alg$-basis of $\Hecke$;
 (b) For any $\lambda\vdash n$, let $\Hecke^{\lambda}$ be the $\Alg$-submodule of $\Hecke$ spanned by all elements $m_{st}$ where $s, t\in \mathbb{T}(\mu)$ for some $\lambda\unlhd\mu$, then $\Hecke^{\lambda}$ is a two-sided ideals in $\Hecke$.
 \end{thm}
 Note that the element that we denote by $T_w$ corresponds to the element $q^{\ell(w)}T_w$ in Murphy's notation. Thus the element denoted by $C_{w_{J(\lambda)}}$ in the above statement is exactly as in Murphy's work, except the associated coefficient $\epsilon_{w_{J(\lambda)}} q^{\ell(w_{J(\lambda)})}$. However, this does not affect the validity of (a) and (b) since $q$ is invertible in $\Alg$.
 The statement in (a) can be found in Murphy~\cite[Th.3.9]{Murphy1} or Murphy~\cite[Th. 4.17]{Murphy2}. The statement(b) is proved in~\cite[Th. 4.18]{Murphy2}.

 Murphy also obtains the following result concerning the Specht modules of $\Hecke$. For any $\lambda\vdash n$, let $\hat{\Hecke^\lambda}$ be the $\Alg$-submodule of $\Hecke$ spanned by all $m_{st}$ where $s, t\in \mathbb{T}(\mu)$ for some $\mu \vdash n$ such that $\lambda\vartriangleleft \mu$. Thus, we have
 \[
 \hat{\Hecke^\lambda}=\sum \limits_{\mu}\Hecke^\mu
\]
where the sum runs over all $\mu \vdash n$ such that $\lambda\vartriangleleft \mu$. In particular, $\hat{\Hecke^{\lambda}}$ is a two-sided ideal and we have $\Hecke^\lambda=\Hecke C_{w_J(\lambda)}\Hecke+\hat{\Hecke^\lambda}$
 \begin{defn}~\cite{Mathas2}
 For $\lambda\vdash n$, the Specht module $S^\lambda$ is defined to be the left $\Hecke$-module $(\hat{\Hecke^{\lambda}}+ C_{w_{J(\lambda)}})\Hecke$.
 \end{defn}
 Note that $\hat{\Hecke^{\lambda}}+ C_{w_{J(\lambda)}}$ is an element of the $\Hecke$-module $\Hecke/\hat{\Hecke^{\lambda}}$ so that $S^\lambda$ is a submodule of $\Hecke/\hat{\Hecke^{\lambda}}$. As we defined it, the Specht module $S^\lambda$ is isomorphic to the dual of the Specht module which Dipper and James~\cite{James} indexed by $\lambda'$.

 For a standard $\lambda$-tableau $t$ let $m_t=m_{tt^\lambda}+\hat{\Hecke^{\lambda}}=T_{d(t)}C_{w_{J(\lambda)}}+\hat{\Hecke^{\lambda}}$ , We have  
 \begin{thm}~\cite{Geck1,Murphy2}
The Specht module $S^\lambda$ is free as an $\Hecke$-module with basis $\{m_t|t\in \mathbb{T}(\lambda)\}$, and $\Hecke^\lambda/\hat{\Hecke^\lambda}$ is a direct sum of
$|\mathbb{T}(\lambda)|$ copies of $S^\lambda$.
 \end{thm}
 While
 \begin{lemma}~\cite{Mathas2}
 Suppose $t\in \mathbb{T}(\lambda)$ such that $i\in I_1(t)$, then for all $s\in \mathbb{T}(\lambda)$
 \[
 T_i m_{st}\equiv q m_{st}+\sum_{v\vartriangleleft s}r_v m_{vt}\qquad\text{ mod $\hat{\Hecke^{\lambda}}$}
 \]
 for some $r_v\in \Alg$.
 \end{lemma}
 \begin{corollary}Let $t\in \mathbb{T}(\lambda)$ and $s=s_i t$ for some $1\leq i\leq n-1$, then
 \[
 \postdisplaypenalty=10000 \advance\abovedisplayskip 0 pt minus 3
pt \advance\belowdisplayskip 0 pt minus 3 pt  T_i m_t=\begin{cases}
m_s, \text{if $i\in I_0(t)$} \\
m_s+(q-q^{-1})m_t, \text{if $i\in I_0(t')$} \\
-q^{-1}m_t, \text{if $i\in I_1(t')$}\\
q m_t+\sum_{v\vartriangleleft t}r_v m_v  \qquad \text{mod $\hat{\Hecke^{\lambda}}$},\text{if $i\in I_1(t)$}.
 \end{cases}
 \] where $r_v\in \Alg$.
\end{corollary}
We apply with Theorem 4.1 and 4.3 to establish the transition between Murphy's basis and $W\!$-graph basis of the Specht module. We also note that in the references, the authors related the Kazhdan-Lusztig cell module and the corresponding Specht module in the case of symmetry group, group algebra and Hecke algebra of type $A$. See Naruse~\cite{Naruse}, Garsia-MacLarnan~\cite{Garsia} and MacDonough  and Pallicaros ~\cite{Pall} ect.
\begin{thm}For a fixed $\lambda\vdash n$, we define the elements of the $C$-basis for $S^{\lambda}$
\begin{align*}
C_{d(s)w_{J(\lambda)}}&=m_s-q\sum_{d(t)< d(s)}p_{t,s}m_t,\\
&=T_{d(s)}C_{w_{J(\lambda)}}-q\sum_{t\vartriangleleft s}p_{t,s}T_{d(t)}C_{w_{J(\lambda)}}\,\,\,\text{mod($\hat{\Hecke}$)}.
\end{align*}
where $s,t\in \mathbb{T}(\lambda)$ and $p_{t,s}\in \mathbb{Z}(q)$ will be defined recursively by
\begin{equation}
 T_iC_{d(t)w_{J(\lambda)}}=\begin{cases}
 -q^{-1}C_{d(t)w_{J(\lambda)}}, \text{if $i\in I(t')$} \\
 qC_{d(t)w_{J(\lambda)}}+\sum\limits_{i\in I(u'), u\vartriangleleft t}\mu(u,t)C_{d(u)w_{J(\lambda)}}, \text{if $i\in I_1(t)$} \\
 qC_{d(t)w_{J(\lambda)}}+C_{s_id(t)w_{J(\lambda)}}+\sum\limits_{i\in I(u'), u\vartriangleleft t}\mu(u,t)C_{d(u)w_{J(\lambda)}}, \text{if $i\in I_0(t)$}
 \end{cases}
 \end{equation}
 where $\mu(u, t)$ is the constant term of the polynomial $p_{u, t}$.
\end{thm}

\end{document}